\documentclass[oneside,english]{amsart}
\usepackage[T1]{fontenc}
\usepackage[latin9]{inputenc}
\usepackage{textcomp}
\usepackage{amstext}
\usepackage{amsthm}
\usepackage{amssymb}

\makeatletter
\numberwithin{equation}{section}
\numberwithin{figure}{section}
\theoremstyle{plain}
\newtheorem{thm}{\protect\theoremname}[section]
\theoremstyle{definition}
\newtheorem{defn}[thm]{\protect\definitionname}
\theoremstyle{definition}
\newtheorem{example}[thm]{\protect\examplename}
\theoremstyle{plain}
\newtheorem{cor}[thm]{\protect\corollaryname}

\makeatother

\usepackage{babel}
\providecommand{\corollaryname}{Corollary}
\providecommand{\definitionname}{Definition}
\providecommand{\examplename}{Example}
\providecommand{\theoremname}{Theorem}

\begin{document}

\title{Estimation of recurrence for nilpotent group action}
\begin{abstract}
We estimate size of recurrence of an action of a nilpotent group by
homeomorphisms of a compact space for polynomial mappings into a nilpotent
group form the partial semigroup $(\mathcal{P}_{f}(\mathbb{N}),\uplus)$.
To do this we have used algebraic structure of the Stone-\v{C}ech
copactification partial semigroup and that of the given nilpotent
group.
\end{abstract}

\author{Aninda Chakraborty}

\address{Department of Mathematics, University of Kalyani}

\email{anindachakraborty2@gmail.com}

\author{Dibyendu De}

\address{Department of Mathematics, University of Kalyani}

\email{dibyendude@gmail.com}

\author{Sayan Goswamy}

\address{Department of Mathematics, University of Kalyani}

\email{sayan92m@gmail.com}
\maketitle

\section{Introduction}

One of the earliest Ramsey theoretic result is celebrated van der
Waerden theorem on arithmetic progressions, which states that if the
set of integers is partitioned into finitely many classes then at
least one of the classes contains arbitrarily long arithmetic progressions.
\begin{thm}
Whenever we partition the set of integers into finitely many cells
then one of the cell will contain arbitrarily long arithmetic progression.
\end{thm}

Furstenberg and Weiss \cite{key-9} offered a new approach, based
on the methods of topological dynamics, to results of this type.
\begin{thm}
Let $(X,d)$ be a compact metric space and let $T_{1},T_{2},\ldots,T_{k}$,
be commuting self-homeomorphisms of $X$: Then for any $\epsilon>0$
there exist $x\in X$ and $n\in\mathbb{N}$ such that $d(T_{i}^{n}x,x)<\epsilon$
for all $i=1,2,\ldots,k$.
\end{thm}

An IP version of the above theorem was also presented in \cite{key-9}.
\begin{thm}
Let $T_{\alpha\in\mathcal{F}}^{(1)},T_{\alpha\in\mathcal{F}}^{(2)},\ldots,T_{\alpha\in\mathcal{F}}^{(k)}$
be an IP-systems in G of commuting self-homeomorphisms of a compact
metric space $(X,d)$. Then for any $\epsilon>0$ there exists $x\in X$
and $n\in\mathbb{N}$ such that $d(T_{\alpha\in\mathcal{F}}^{(i)}x,x)<\epsilon$
for all $i=1,2,\ldots,k$.
\end{thm}

A combinatorial version of the above theorem is the following.
\begin{thm}
Let G be an abelian group, and let $T_{\alpha\in\mathcal{F}}^{(1)},T_{\alpha\in\mathcal{F}}^{(2)},\ldots,T_{\alpha\in\mathcal{F}}^{(k)}$,
be IP-systems in $G$. For any finite coloring of $G$ there exist
$h\in G$ and a nonempty $\alpha\in\mathcal{F}$ such that the elements
$hT_{\alpha\in\mathcal{F}}^{(1)},hT_{\alpha\in\mathcal{F}}^{(2)},\ldots,hT_{\alpha\in\mathcal{F}}^{(k)}$
all have the same color.
\end{thm}

In \cite[Theorem 4.1]{key-7} Bergelson and Leibman established a
nil-IP-multiple recurrence theorem which extended all the abelian
results mentioned above to a nilpotent setup.
\begin{thm}
Let $G$ be a nilpotent group of self-homeomorphisms of a compact
metric space $(X,d)$ and let $P_{1},\ldots,P_{k}:\mathcal{F}\to G$
be polynomial mappings satisfying $P_{i}(\emptyset)=1_{G}$ for all
$i=1,2,\ldots,k$. Then for any $\epsilon>0$ there exist $x\in X$
and $\alpha\in\mathcal{F}$ such that $d(P_{i}(\alpha)x,x)<\epsilon$
for all $i=1,2,\ldots,k$.
\end{thm}

In the present work we will prove that the collection of $\alpha's$
in the partial semigroup (to be defined in the next section) $(\mathcal{P}_{f}(\mathbb{N}),\uplus)$
will be an IP$^{*}$-set. Our machinery in the proof will be using
algebraic structure of Stone-\v{C}ech compactification of discrete
semigroup. In fact we are inspired by \cite{key-11} and \cite{key-12}.

\section{Preliminaries}

For our purpose let us first introduce a brief algebraic structure
of $\beta S$ for a discrete semigroup $(S,+)$. We take the points
of $\beta S$ to be the ultrafilters on $S$, identifying the principal
ultrafilters with the points of $S$ and thus pretending that $S\subseteq\beta S$.
Given$A\subseteq S$ let us set, 
\[
\overline{A}=\{p\in\beta S:A\in p\}.
\]
 Then the set $\{\overline{A}:A\subseteq S\}$ is a basis for a topology
on $\beta S$. The operation$+$ on $S$ can be extended to the Stone-\v{C}ech
compactification $\beta S$ of $S$ so that $(\beta S,+)$ is a compact
right topological semigroup (meaning that for any $p\in\beta S$,
the function $\rho_{p}:\beta S\rightarrow\beta S$ defined by $\rho_{p}(q)=q+p$
is continuous) with $S$ contained in its topological center (meaning
that for any $x\in S$, the function$\lambda_{x}:\beta S\rightarrow\beta S$
defined by $\lambda_{x}(q)=x+q$ is continuous). Given $p,q\in\beta S$
and $A\subseteq S$, $A\in p+q$ if and only if $\{x\in S:-x+A\in q\}\in p$,
where $-x+A=\{y\in S:x+y\in A\}$. 

A nonempty subset $I$ of a semigroup $(T,+)$ is called a left ideal
of $\emph{T}$ if $T+I\subset I$, a right ideal if $I+T\subset I$,
and a two sided ideal (or simply an ideal) if it is both a left and
right ideal. A minimal left ideal is the left ideal that does not
contain any proper left ideal. Similarly, we can define minimal right
ideal and smallest ideal.

Any compact Hausdorff right topological semigroup $(T,+)$ has a smallest
two sided ideal

\[
\begin{array}{ccc}
K(T) & = & \bigcup\{L:L\text{ is a minimal left ideal of }T\}\\
 & = & \,\,\,\,\,\bigcup\{R:R\text{ is a minimal right ideal of }T\}
\end{array}
\]

Given a minimal left ideal $L$ and a minimal right ideal $R$, $L\cap R$
is a group, and in particular contains an idempotent. An idempotent
in $K(T)$ is called a minimal idempotent. If $p$ and $q$ are idempotents
in $T$, we write $p\leq q$ if and only if $p+q=q+p=p$. An idempotent
is minimal with respect to this relation if and only if it is a member
of the smallest ideal. See \cite{key-13}  for an elementary introduction
to the algebra of $\beta S$ and for any unfamiliar details.
\begin{defn}
(Partial semigroup) Let ${G}$ be a set, let ${X\subset G\times G}$
and let ${\ast:X\rightarrow G}$ be an operation. The triple ${(G,X,\ast)}$
is a partial semigroup if it satisfies the following properties:

For any ${x,y,z\in G}$, then ${(x\ast y)\ast z=x\ast(y\ast z)}$
in the sense that either both sides are undefined or both are defined
and equal. For any ${n\in{\mathbb{N}}}$ and ${x_{1},\dots,x_{n}\in G}$
there exists ${y\in G\setminus\{x_{1},\dots,x_{n}\}}$ such that ${x_{i}\ast y}$
is defined. Observe that the second condition implies that ${G}$
is infinite. A partial semigroup is commutative if ${x\ast y=y\ast x}$
for every ${(x,y)\in X}$.
\end{defn}

\begin{example}
Let ${G=\mathcal{F}({\mathbb{N}})}$, let ${X:=\{(\alpha,\beta)\in G^{2}:\alpha\cap\beta=\emptyset\}}$
be the family of all pairs of disjoint sets, and let ${\ast:X\rightarrow G}$
be the union. It is easy to check that this is a commutative partial
semigroup. We shall denote this partial semigroup as $(\mathcal{F}({\mathbb{N}}),\uplus)$.
\end{example}

To define a nil IP-polynomial we need to recall some facts. Let us
denote the collection of finite subsets of $\mathbb{N}$ be $\mathcal{F}$. 
\begin{defn}
For a semigroup $G$, an IP-system is a mapping $\mathcal{F}\rightarrow G$
defined as $\alpha\rightarrow g_{\alpha}$, so that $g_{\alpha\cup\beta}=g_{\alpha}g_{\beta}$
provided that $\alpha\cap\beta=\emptyset$.
\end{defn}

However if the above $G$ is commutative in nature, then we can define
a IP-polynomial inductively as given below.
\begin{defn}
Assuming a IP-polynomial of degree 0 be constant, we say that an IP-polynomial
defined as $P:\mathcal{F}\rightarrow G$, of degree $\leq d$ if for
any $\beta\in\mathcal{F}$, there exists a polynomial mapping $D_{\beta}P:\mathcal{F}(\mathbb{N}\setminus\beta)\rightarrow G$
of degree $\leq d-1$ such that $P(\alpha\cup\beta)=P(\alpha)+(D_{\beta}P)(\alpha)$
$\forall\alpha\in\mathcal{F}(\mathbb{N}\setminus\beta)$. 
\end{defn}

In the case of abelian $G$, it was proved in \cite[Theorem 8.1]{key-5},
that the necessary and sufficient condition for a mapping $P:\mathcal{F}\rightarrow G$
to be an IP-polynomial is that there must exists a $d\in\mathbb{N}$
and a family $\{g_{(j_{1},j_{2},\ldots,j_{d})}\}_{(j_{1},j_{2},\ldots,j_{d})\in\mathbb{N}^{d}}$of
elements in $G$, such that for any $\alpha\in\mathcal{F}$, one has
$P(\alpha)=\prod_{(j_{1},j_{2},\ldots,j_{d})\in\alpha^{d}}g_{(j_{1},j_{2},\ldots,j_{d})}.$
It is a characterization of commutative IP-polynomials which can be
generalized in case of nilpotent groups.
\begin{defn}
For a nilpotent group $G$, we define a mapping $P:\mathcal{F}\rightarrow G$
to be a nil IP-polynomial if for some $d\in\mathbb{N}$, there exists
a family $\{g_{(j_{1},j_{2},\ldots,j_{d})}\}_{(j_{1},j_{2},\ldots,j_{d})\in\mathbb{N}^{d}}$of
elements in $G$ and a linear order $\prec$ on $\mathbb{N}^{d}$
such that for any $\alpha\in\mathcal{F}$ we have, $P(\alpha)=\prod_{(j_{1},j_{2},\ldots,j_{d})\in\alpha^{d}}^{\prec}g_{(j_{1},j_{2},\ldots,j_{d})}.$
\end{defn}

V. Bergelson and Leibman in \cite[Theorem 4.1]{key-7} proved the
following result:
\begin{thm}
Let $G$ be a nilpotent group of self-homeomorphisms of a compact
metric space $(X,\rho).$ For any weight (defined in next section)
$\omega\in W,$ any $k\in\mathbb{N}$, and any $\epsilon>0$, $\exists N\in\mathbb{N}$
such that if $S$ is a set of cardinality $\geq N$ and $\mathcal{A}$
is a system of $k$ polynomial mappings $\mathcal{F}(S)\rightarrow G$
satisfying $\omega(P)\leq\omega$ and $P(\emptyset)=1{}_{G},\,P\in\mathcal{A}$,
there exists a point $x\in X$ and a nonempty $\alpha\in\mathcal{F}(S)$
such that $\rho(P(\alpha)x,x)<\epsilon$ $\forall P\in\mathcal{A}$.
\end{thm}

In this paper we will prove the following theorem which is algebraic
version of the above theorem:
\begin{thm}
\label{alkgeb nil}Let $\mathcal{R}$ be a system and $v=v+v\in\beta\mathcal{F}$
and let $A$ be a picewise syndetic subset of $G$ and let $L$ be
a minimal left ideal of $\beta G$ such that $\overline{A}\cap L\neq\emptyset$.
Then, 
\[
\{\alpha\in\mathcal{F}:\overline{A}\cap L\cap_{p\in\mathcal{R}}\overline{P(\alpha)^{-1}A}\neq\emptyset\}\in v.
\]
\end{thm}

And we will apply this theorem present the following theorem which
is the more refined version of Nilpotent PHJ given in \cite{key-7}\cite{key-1}
\begin{thm}
\label{nil top}Let $X$ be a compact metric space and $G$ be the
nilpotent group of self homeomorphisms acting on $X$ . Let $\mathcal{A}$
be a system consisting of polynomials $\{P_{1},P_{2},\ldots,P_{k}\}$.
Then for every $x\in X$, $\varepsilon>0$, there exists some $a\in G$
such that 
\[
\textup{\ensuremath{\{\alpha\in\mathcal{F}:\rho(T^{P_{i}(\alpha)}ax,ax)<\varepsilon\}}}
\]
 is an IP$^{*}$ set in the partial semigroup $(\mathcal{P}_{f}(\mathbb{N}),\uplus)$.
\end{thm}

\section{Polynomial mappings and triangular monomials}

In this section we will follow \cite{key-7}. Let $G_{1}\geq G_{2}\geq\ldots\geq G_{l}\geq G_{l+1}\geq\ldots$
be the lower central series of a nilpotent group $G$, and as $G$
is nilpotent there must exists some $n\in\mathbb{N}$ such that $G_{n}=\{1_{G}\}$.

Let $\mathcal{F}^{\leq d}(S)$ be the set of all subsets of S of cardinality
$\leq d.$
\begin{defn}
Let $S$ be a nonempty set.\textbf{ }For a Nilpotent Group $G$, Monomial
of degree $d$ on $S$ with values in $G$ is a pair $(u,\prec)$,
where $u:S^{d}\rightarrow G,$ is a mapping and $\prec$ is a linear
order on $S^{d}$. The monomial $(u,\prec)$ induces a monomial mapping
$P_{u}:\mathcal{F}(S)\rightarrow G$ by the rule $P_{u}(\alpha)=\prod_{s\in\alpha^{d}}^{\prec}u(s)$.
\end{defn}

The level of a monomial $(u,\prec)$ of degree $d$ is defined to
be the positive integer $l$ such that, $u(S^{d})\subseteq G_{l}\setminus G_{l+1}$.
For the nilpotency class $q$ of $G$, we may define the level of
the identity monomial $u(S^{d})=1_{G}$is $q+1$. We can endow the
set of weight $W$ of monomials. by lexicographic ordering.
\begin{defn}
A polynomial mapping $P:\mathcal{F}(S)\rightarrow G$ is defined to
be the product of finitely many monomial mappings as $P(\alpha)=P_{u_{1}}(\alpha)P_{u_{2}}(\alpha)\ldots P_{u_{m}}(\alpha),\alpha\in\mathcal{F}(S)$
where $P_{u_{1}},P_{u_{2}},\ldots,P_{u_{m}}$are monomial mappings
corresponding to $(u_{1},\prec_{1})$, $\ldots$$,(u_{m},\prec_{m})$.
And the weight $w(P)$ is taken to be minimum over the possible all
representations of $P$ as the product $P=P_{u_{1}}P_{u_{2}}\ldots P_{u_{m}}$of
monomial mappings.
\end{defn}

As,it is not guaranteed for a polynomial $P$ of weight $(l,d)$,
$P(\mathcal{F}(S))\subseteq G_{l}\setminus G_{l+1}$, $(P(\mathcal{F}(S))\subseteq G_{l}$
but not guaranteed in $G_{l+1}$ the triangular monomial was introduced.

\medskip{}

\begin{defn}
If $v:\mathcal{F}^{=d}(S)\rightarrow G$ is a mapping and $\prec$
is a linear order on $\mathcal{F}^{=d}(S)$, then the pair$(v,\prec)$
iscalled a triangular monomial of degree $d$. This induces a polynomial
mapping $P:\mathcal{F}(S)\rightarrow G$ by the rule 
\[
P_{v}(\alpha)=\prod_{t\in\mathcal{F}^{=d}(\alpha)}^{\prec}v(t).
\]
\end{defn}

\medskip{}

Now as \cite{key-7} we can represent a polynomial mapping $P$ in
the form $P=P_{d}P_{d-1}\ldots P_{0}Q,$ where each $P_{i},i=0,\ldots,d$
is a polynomial mapping induced by the triangular monomial $v_{i},i=0,\ldots,d$
and $Q$ is the polynomial mappings of higher degree.

Now using the triangular monomials in \cite[Section 3.1]{key-7},
the weight of a polynomial is defined as:
\begin{defn}
Let $P:\mathcal{F}(S)\rightarrow G$ be a polynomial mapping. The
weight of $P$ is defined to be the pair $(l,d)$, whenever $P$ has
a representation $P=P_{v}Q,$ where $P_{v}$ is the monomial mapping
induced by the triangular monomial $v$, $w(v)=(l,d)$ and $w(Q)<(l,d).$
If $\varphi:G_{l}\rightarrow G_{l}\slash G_{l+1}$, then we call $\varphi\circ v:S^{d}\rightarrow G_{l}\slash G_{l+1}$
the principal part of $P$ and denote it by $M(P)$. We define the
$\sim$ relation as, $P\sim P^{'}$ iff $w(P)=w(P^{'})$ and the principal
parts coincides.
\end{defn}

However the following definitions and the reduction of weight of a
system follows from \cite{key-7}.
\begin{defn}
Let us denote by $W$ the set of weights of polynomials $\mathcal{F}(S)\to G$;
that is, the set of pairs $(l,d)$ with $l,d\in\mathbb{Z}$, $1\leq l\leq q$;
$d\geq0$. Let $\mathcal{A}$ be a system; the weight vector $w(\mathcal{A})$
of $\mathcal{A}$ is a function $w(\mathcal{A}):W\to\{0,1.2,\ldots\}$
defined by the number of equivalence classes of polynomial mappings
of weight $w$ having its representatives in $\mathcal{A}$.

We order the weight vector lexicographically: $w(\mathcal{A})<w(\mathcal{A}^{\prime})$
iff or some $w\in W$ one has $w(\mathcal{A})(w)<w(\mathcal{A}^{\prime})(w)$
and $w(\mathcal{A})(w^{\prime})=w(\mathcal{A}^{\prime})(w^{\prime})$
for all $w^{\prime}>w$. We say that a system $\mathcal{A}$ precedes
a system $\mathcal{A}^{\prime}$ if $w(\mathcal{A})<w(\mathcal{A}^{\prime})$. 
\end{defn}

For any nil IP polynomial $P:\mathcal{F}(S)\rightarrow G$, if the
weight $w(P)=(l,d)$ then for any $g\in G,$ as $g^{-1}Pg=P[P,g]$,
we have $w(g^{-1}Pg)\leq w(P)$
\begin{defn}
Let $\gamma\in\mathcal{F}(S)$, let $P$ be a mapping $\mathcal{F}(S)\to G$;
We define $U_{\gamma}P:\mathcal{F}(S\setminus\gamma)\to G$ by
\[
U_{\gamma}P(\alpha)=P(\gamma\cup\alpha).
\]
\end{defn}

Now we state the following theorems from \cite{key-7} of our interest:
\begin{thm}
Let $S$ be a set, let $G$ be a nilpotent group and $\mathcal{A}$
be a system of polynomial mappings $\mathcal{F}(S)\to G$. Then the
following holds:

1.$w(P^{-1}U_{\gamma}P)<w(P)$.

2. If $\gamma\in\mathcal{F}(S)$ and a system$\mathcal{A}^{,}$of
polynomial mappings $\mathcal{F}(S\setminus\gamma)\to G$ is such
that each element of $\mathcal{A}^{\prime}$is equivalent to $P_{\mathcal{F}(S\setminus\gamma)}$
for some $P\in\mathcal{A}$, then $w(\mathcal{A}^{\prime})\leq w(A).$

3. If $\mathcal{A}^{\prime}$consists of polynomial mappings of the
form $PQ\,\text{ and }\,Q$ where $P\in\mathcal{A}$ and $Q$ is a
polynomial mapping $\mathcal{F}(S)\to G$ then $w(\mathcal{A}^{\prime})\leq w(A).$

4. If $\mathcal{A}^{\prime}$consists of polynomial mappings of the
form $PQ\,\text{ and }\,QP$ where $P\in\mathcal{A}$ and $Q$ is
a polynomial mapping $\mathcal{F}(S)\to G$ with $\omega(Q)<\omega(P)$
then $w(A^{,})\leq w(A).$

5. Let $Q\in\mathcal{A}$ be a nontrivial polynomial mapping with
$\omega(Q)\leq\omega(P)\forall P\in\mathcal{A}.$ If $\mathcal{A}^{\prime}$is
a system of polynomial mappings of the form $Q^{-1}P$ and $P^{-1}Q,$
then $\omega(\mathcal{A}^{\prime})<\omega(A).$
\end{thm}

A parallel version of polynomial system in nilpotent group named as
$VIP(G\centerdot)$ polynomial which capture a polynomial as a prefiltration
of the group $G$ and it has a weight defined in another way presents
in \cite{key-1}. Though we are using the technique to handle polynomials
as it in \cite{key-7} but it can be similarly done by \cite{key-1}
for $VIP(G\centerdot)$ polynomials and the theorems 2.5, 2.6 hold
for those polynomials also.
\begin{proof}[Proof of Theorem \ref{alkgeb nil}]
$\,$ If possible the Theorem is not true. Let $\mathcal{R}$ be
the minimal among all counterexamples. Then, 
\begin{equation}
D=\mathcal{F}\setminus\{\alpha\in\mathcal{F}:\overline{A}\cap L\cap_{p\in\mathcal{R}}\overline{P(\alpha)^{-1}A}\neq\emptyset\}\in v.
\end{equation}

Due to picewise syndeticity of $A$, we can choose $q_{0}\in L,$
such that $A\in q_{0}.$ Then $B=\{\gamma\in G:\gamma A\in q_{0}\}$
is syndetic. Therefore we can find a finite $H\in\mathcal{F}(G)$,
such that $G\subseteq\cup_{t\in H}t^{-1}B$. Let us pick $t_{0}\in H$
such that $C_{0}=t_{0}^{-1}B\in q_{0}$. Pick, $Q\in\mathcal{R}$
is of minimal degree polynomial. And take $\mathcal{A}=\{t_{0}^{-1}PQ^{-1}t_{0}:P\in\mathcal{R}\}.$
Then $w(\mathcal{A})<w(\mathcal{R}).$ Let us set 
\[
E_{0}=\{\alpha\in\mathcal{F}:\overline{C_{0}}\cap L\cap_{R\in\mathcal{A}}\overline{R(\alpha)^{-1}C_{0}}\neq\emptyset\}\in v.
\]
 Since $w(\mathcal{A})<w(\mathcal{R})$ and $\mathcal{R}$ is minimal
among the counter examples, so $E_{0}\in v$. 

Let us pick $\delta_{1}\in E_{0}\cap D^{*}$ and pick 

\begin{equation}
r_{1}\in\overline{C_{0}}\cap L\cap_{R\in\mathcal{A}}\overline{R(\delta_{1})^{-1}C_{0}}.
\end{equation}

Let us take, $q_{1}=\big([t_{0},Q^{-1}(\delta_{1})]Q(\delta_{1})\big){}^{-1}r_{1}$. 

And so from (3.2), $R(\delta_{1})^{-1}C_{0}\in r_{1}$ so that $(t_{0}^{-1}PQ^{-1}(\delta_{1})t_{0})^{-1}C_{0}\in r_{1}$.
This implies that 
\[
\begin{array}{cc}
 & t_{0}^{-1}Q(\delta_{1})P(\delta_{1})^{-1}t_{0}C_{0}\in r_{1}\\
\Rightarrow & t_{0}^{-1}Q(\delta_{1})P(\delta_{1})^{-1}t_{0}C_{0}\in r_{1}\\
\Rightarrow & t_{0}^{-1}Q(\delta_{1})P(\delta_{1})^{-1}t_{0}t_{0}^{-1}B\in r_{1}\\
\Rightarrow & t_{0}^{-1}Q(\delta_{1})P(\delta_{1})^{-1}t_{0}t_{0}^{-1}B\in r_{1}\\
\Rightarrow & t_{0}^{-1}Q(\delta_{1})P(\delta_{1})^{-1}B\in r_{1}\\
\Rightarrow & [t_{0},Q^{-1}(\delta_{1})]Q(\delta_{1})t_{0}^{-1}P(\delta_{1})^{-1}B\in r_{1}\\
\Rightarrow & \cap_{P\in R}t_{0}^{-1}P(\delta_{1})^{-1}B\in q_{1}.
\end{array}
\]

Now choose 
\begin{equation}
t_{1}^{-1}B\in q_{1}.
\end{equation}

Also choose, 
\[
\mathcal{B}=\{t_{1}^{-1}PQ^{-1}t_{1},\,t_{0}^{-1}P(\delta_{1})^{-1}(U_{\delta_{1}}P)t_{0}t_{1}^{-1}Q^{-1}t_{1}:P\in\mathcal{R}\}
\]

and so, $w(\mathcal{B})\prec w(\mathcal{R})$.

Choose 
\begin{equation}
C_{1}=t_{1}^{-1}B\bigcap\cap_{p\in\mathcal{R}}t_{0}^{-1}P(\delta_{1})^{-1}B
\end{equation}

And from (3.3) and (3.4) $C_{1}\in q_{1}$.

Let 
\[
E_{1}=\{\alpha\in\mathcal{F}:\overline{C_{1}}\cap L\cap_{R\in\mathcal{B}}\overline{R(\alpha)^{-1}C_{1}}\neq\emptyset\}\in v.
\]
 Since $w(\mathcal{B})<w(\mathcal{R})$. And choose 

\begin{equation}
\delta_{2}\in D^{*}\cap(-\delta_{1}\cup D^{*})\cap E_{1}
\end{equation}

So, $r_{2}\in\overline{C_{1}}\cap L\cap_{R\in\mathcal{B}}\overline{R(\delta_{2})^{-1}C_{1}}$
and let, $q_{2}=([t_{1},Q^{-1}(\delta_{2})]Q(\delta_{2}))^{-1}r_{2}.$

Now$t_{1}^{-1}PQ^{-1}t_{1}\in\mathcal{B}$ implies that

\begin{equation}
\begin{array}{cccc}
 & (t_{1}^{-1}PQ^{-1}(\delta_{2})t_{1})^{-1}C_{1} & \in & r_{2}\\
\Rightarrow & (t_{1}^{-1}Q(\delta_{2})P(\delta_{2})^{-1}t_{1})t_{1}^{-1}B & \in & r_{2}\\
\Rightarrow & t_{1}^{-1}Q(\delta_{2})P(\delta_{2})^{-1}B & \in & r_{2}\\
\Rightarrow & [t_{1}Q(\delta_{2})^{-1}]Q(\delta_{2})t_{1}^{-1}P(\delta_{2})^{-1}B & \in & r_{2}\\
\Rightarrow & t_{1}^{-1}P(\delta_{2})^{-1}B & \in & ([t_{1},Q^{-1}(\delta_{2})]Q(\delta_{2}))^{-1}r_{2}=q_{2}\\
\Rightarrow & \cap_{P\in\mathcal{R}}t_{1}^{-1}P(\delta_{2})^{-1}B & \in & q_{2}
\end{array}
\end{equation}

Again

\begin{equation}
\begin{array}{cccc}
 & t_{0}^{-1}P(\delta_{1})^{-1}(U_{\delta_{1}}P)t_{0}t_{1}^{-1}Q^{-1}t_{1} & \in & \mathcal{B}\\
\Rightarrow & ((t_{0}^{-1}P(\delta_{1})^{-1}(U_{\delta_{1}}P)t_{0}t_{1}^{-1}Q^{-1}t_{1})(\delta_{2}))^{-1}C_{1} & \in & r_{2}\\
\Rightarrow & t_{1}^{-1}Q(\delta_{2})t_{1}t_{0}^{-1}P(\delta_{1}\cup\delta_{2})^{-1}P(\delta_{1})t_{0}t_{0}^{-1}P(\delta_{1})^{-1}B & \in & r_{2}\\
\Rightarrow & t_{1}^{-1}Q(\delta_{2})t_{1}t_{0}^{-1}P(\delta_{1}\cup\delta_{2})^{-1}B & \in & r_{2}\\
\Rightarrow & [t_{1},Q^{-1}(\delta_{2})]Q(\delta_{2})t_{0}^{-1}P(\delta_{1}\cup\delta_{2})^{-1}B & \in & r_{2}\\
\Rightarrow & t_{0}^{-1}P(\delta_{1}\!\cup\!\delta_{2})^{-1}B & \in & q_{2}\\
\Rightarrow & \cap_{P\in R}t_{0}^{-1}P(\delta_{1}\cup\delta_{2})^{-1}B & \in & q_{2}.
\end{array}
\end{equation}

Now we can choose iteratively, 

(1) $t_{j}^{-1}B\in q_{j}\;j\in\{0,1,\ldots,m\}$

(2) $\delta_{l}\cup\delta_{l+1}\cup\ldots\cup\delta_{m}\in D^{*}$
for $l\in\{0,1,\ldots,m\}$

(3) $t_{l}^{-1}P(\delta_{l}\cup\delta_{l+1}\cup\ldots\cup\delta_{m})^{-1}B\in q_{m}\;\forall P\in R\;\text{ and }\:l\in\{0,1,\ldots,m-1\}$

Now since, $H\in\mathcal{P}_{f}(G)$ so we may choose $l<m$ and $t_{l}=t_{m}$
and put $\delta=\delta_{l}\cup\delta_{l+1}\cup\ldots\cup\delta_{m}\in D^{*}$.

Take $a\in t_{m}^{-1}B\cap_{p\in R}t_{m}^{-1}P(\delta)^{-1}B\in q_{m}.$

This implies $t_{m}a\in B\cap_{p\in R}P(\delta)^{-1}B.$ 

Now,

\begin{equation}
t_{m}a\in B\Rightarrow(t_{m}a)^{-1}A\in q_{0}\Rightarrow A\in t_{m}aq_{0}
\end{equation}

And, 
\begin{equation}
t_{m}a\in P(\delta)^{-1}B\Rightarrow P(\delta)t_{m}a\in B\Rightarrow(P(\delta)t_{m}a)^{-1}A\in q_{0}\Rightarrow P(\delta)^{-1}A\in t_{m}aq_{0}
\end{equation}

From (3.8) and (3.9) it is clear that $t_{m}aq_{0}\in\overline{A}\cap L\cap_{p\in R}\overline{P(\delta)^{-1}A}$
but this contradicts $\delta\in D.$

Hence the theorem is proved.
\end{proof}
Now we are in situation to prove the theorem 2.8.
\begin{proof}[Proof of theorem \ref{nil top}]
 Let $X$ be a compact metric space with action of $G,$ a nilpotent
group.

Let us choose an $\varepsilon>0,$ and $x\in X$. Let $Y=\overline{Gx}$
be the orbit closure of $x$ in $X$. Then $Y$ is itself closed and
hence compact.

Let us take $V_{1},V_{1},\ldots,V_{m}$ be balls of radius $\varepsilon$
covering$Y$. Let us give those $g$'s color $i$ for which $gx\in V_{i}$.
This gives a partition of \$G\$ as $G=\cup_{i=1}^{r}C_{j}$ ($r\leq m)$.
Take some $C_{j}$ which is picewise syndetic.

Then using theorem \ref{alkgeb nil}, we obtain that there exists
some $a\in G,$ such that $\{a,P_{1}(\alpha)a,\ldots,P_{k}(\alpha)a\}\subset V_{j}$,
where for all such $a\in G,$ the collection of $\alpha$ is IP$^{*}$
. Therefore we have that 

\[
\text{diam}\{T^{a}x,T^{P_{1}(\alpha)a}x,\ldots,T^{P_{k}(\alpha)a}x\}<\epsilon.
\]
.

Then we see that, $\rho(T^{P_{i}(\alpha)}ax,ax)<\varepsilon$ for
all $i\in\{1,2,\ldots,k\}$ where the collection of $\alpha$ is IP$^{*}$.

So for every $x\in X$. $\varepsilon>0$ we get $\rho(T^{P_{i}(\alpha)}ax,ax)<\varepsilon$
and the collection of $\alpha$ is an IP$^{*}$ set. 
\end{proof}
We now state two corollaries of the theorem \ref{alkgeb nil} which
can be obtained as in \cite{key-11}\cite{key-12}.
\begin{cor}
Let $(G,.)$ be a nilpotent group, let $R\in\mathcal{R},$ and let
$\langle\alpha_{n}\rangle_{n=1}^{\infty}$be a sequence in $\mathcal{F}$
such that $\alpha_{n}<\alpha_{n+1}\forall n.$ If $A$ be a piecewise
syndetic subset of $G,$ then there exists $r\in\bar{A}\cap K(\beta G)$
and $\beta\in FU(\langle\alpha_{n}\rangle_{n=1}^{\infty})$ such that
$\{p(\beta)r:\,p\in R\}\subseteq\bar{A}$.
\end{cor}

The second one is the following:
\begin{cor}
Let $(G,.)$ be a nilpotent group, let $R\in\mathcal{R},$ and let
$\langle\alpha_{n}\rangle_{n=1}^{\infty}$be a sequence in $\mathcal{F}$
such that $\alpha_{n}<\alpha_{n+1}\forall n.$ If $A$ be a piecewise
syndetic subset of $G,$ then there exists $\beta\in FU(\langle\alpha_{n}\rangle_{n=1}^{\infty})$
such that $\{a\in A\,:\,\{p(\beta)a:\,p\in R\}\subseteq A\}$ is piecewise
syndetic.
\end{cor}

\end{document}